\documentclass[12pt]{amsart}
\usepackage[colorlinks=true,citecolor=magenta,linkcolor=magenta]{hyperref}
\usepackage{amssymb,verbatim,amscd,amsmath,graphicx}
\usepackage{enumerate}
\usepackage{graphicx}
\usepackage{caption}
\usepackage{subcaption}
\usepackage{pdfsync}
\usepackage{color,colortbl}
\usepackage{fullpage}
\usepackage{bbm}
\numberwithin{equation}{section}
\newtheorem{theorem}{Theorem}[section]
\newtheorem{lemma}[theorem]{Lemma}

\theoremstyle{definition}
\newtheorem{definition}[theorem]{Definition}

\newtheorem{def-prop}[theorem]{Definition-Proposition}

\newtheorem{example}[theorem]{Example}

\newtheorem*{acknowledgement}{Acknowledgements}
\newtheorem{notation}[theorem]{Notation}
\newtheorem{question}[theorem]{Question}

\newtheorem{algorithm}[theorem]{Algorithm}

\DeclareMathOperator{\supp}{supp}

\DeclareMathOperator{\fideal}{FI}

\newcommand{\ZZ}{{\mathbb Z}}

\newcommand{\kk}{{\mathbbm k}}

\definecolor{orange}{rgb}{1,0.5,0}
\definecolor{violet}{rgb}{0.5,0,1}

\newcommand{\til}{\raise.17ex\hbox{$\scriptstyle\mathtt{\sim}$}}

\def\I{{\mathcal I}}

\def\M{{\mathcal M}}
\def\A{{\mathcal A}}

\def\F{{\mathcal F}}
\def\G{{\mathcal G}}

\def\UP{\mathcal {UP}}
\def\LP{\mathcal {LP}}

\def\1{{\bf 1}}
\def\0{{\bf 0}}

\begin{document}
	
\title{Density of $f$-ideals and $f$-ideals in mixed small degrees}

\author{Huy T\`ai H\`a}
\address{Tulane University \\ Department of Mathematics \\
	6823 St. Charles Ave. \\ New Orleans, LA 70118, USA}
\email{tha@tulane.edu}

\author{Graham Keiper}
\address{McMaster University \\ Department of Mathematics and Statistics \\
1280 Main Street West \\ Hamilton, Ontario, L8S 4K1, Canada}
\email{keipergt@mcmaster.ca}

\author{Hasan Mahmood}
\address{Government College University Lahore \\ Department of Mathematics \\
\\ Lahore, Punjab 54000, Pakistan}
\email{hasanmahmood@gcu.edu.pk}

\author{Jonathan L. O'Rourke}
\address{Tulane University \\ Department of Mathematics \\
	6823 St. Charles Ave. \\ New Orleans, LA 70118, USA}
\email{jorourk2@tulane.edu}



\begin{abstract}
A squarefree monomial ideal is called an $f$-ideal if its Stanley-Reisner and facet simplicial complexes have the same $f$-vector. We show that $f$-ideals generated in a fixed degree have asymptotic density zero when the number of variables goes to infinity. We also provide novel algorithms to construct $f$-ideals generated in small degrees.
\end{abstract}

\maketitle



\section{Introduction} \label{sec.intro}

Let $\kk$ be a field and let $R = \kk[x_1, \dots, x_n]$ be a polynomial ring over $\kk$. For a squarefree monomial ideal $I \subseteq R$ one can construct the \emph{Stanley-Reisner complex} $\Delta_I$ (see \cite{StanleyBook}), whose minimal nonfaces correspond to the minimal generators of $I$, and the \emph{facet complex} $\Delta(I)$ (see \cite{faridi}), whose maximal faces correspond to the minimal generators of $I$. The ideal $I$ is called an \emph{$f$-ideal} if $\Delta_I$ and $\Delta(I)$ have the same \emph{$f$-vector}, which enumerates faces in each dimension. The notion of $f$-ideals was first introduced by Abbasi, Ahmad, Anwar and Baig \cite{deg2}. Properties of, characterizations for, and algorithms constructing $f$-ideals have been much investigated in recent years (cf. \cite{degd, Adam, tswu2, tswuPubl, tswuPub3, fgraph, fsimp, fnote}). Interestingly the concept of $f$-ideals has recently been generalized by \cite{qfi1, qfi2} through the introduction of \emph{quasi} $f$-ideals. In this paper we examine the asymptotic density of $f$-ideals as the number of variables becomes large in addition to providing new algorithms to construct $f$-ideals of small and mixed degrees.

It is known that $f$-ideals minimally generated in a single degree $d$ have exactly $\frac{1}{2}{n \choose d}$ minimal generators (see, for example, \cite[Theorem 3.3]{degd}). Let $\fideal(n,d)$ denote the collection of $f$-ideals in $R$ generated in degree $d$ and let $\I(n,d)$ denote the collection of squarefree monomial ideals in $R$ minimally generated by exactly $\frac{1}{2}{n \choose d}$ generators of degree $d$. We show that for a fixed degree $d$, $\fideal(n,d)$ has asymptotic density 0 in $\I(n,d)$ as $n \rightarrow \infty$. Specifically, we prove the following theorem.

\medskip

\noindent{\bf Theorem \ref{thm.density}.} For a fixed degree $d$, we have
$$\lim_{n \rightarrow \infty} \dfrac{\#\fideal(n,d)}{\#\I(n,d)} = 0.$$

\medskip

To prove Theorem \ref{thm.density}, we make use of a characterization of $f$-ideals in terms of \emph{lower perfect} and \emph{upper perfect} sets given in \cite[Theorem 2.3]{tswuPubl}. In particular, we show that ideals generated by lower perfect sets with the correct number of generators have asymptotic density 1 and ideals generated by upper perfect sets with the correct number of generators have asymptotic density 0 in $\I(n,d)$; see Theorem \ref{thm.ULP}.

In \cite{tswuPubl}, an algorithm to construct $f$-ideals generated in a single degree --- referred to as \emph{pure} $f$-ideals --- was given. However, as pointed out by our density result, Theorem \ref{thm.density}, pure $f$-ideals become rare as the dimension of the ring becomes large. In fact, even with the algorithm presented in \cite{tswuPubl}, it is not clear that for a given pair $(n,d)$, $f$-ideals of degree $d$ in $n$ variables exist; see, for example, \cite[Proposition 3.3]{tswu2} for existence results. Furthermore, until now there were no algorithms to construct $f$-ideals generated in various different degrees --- referred to as \emph{mixed} $f$-ideals. We shall provide some algorithms to construct mixed $f$-ideals generated in degree 2 and 3; see Algorithms \ref{23alg.odd1}, \ref{23alg.odd2}, \ref{23alg.even}, \ref{alg.odd3} and \ref{alg.odd2}. We shall also provide new algorithms to construct pure $f$-ideals in small degrees; particularly, in degrees 3, 4 and 5  see Algorithms \ref{alg.deg3}, \ref{alg.deg4} and \ref{alg.deg5}. Finally we note that by a result of Budd and Van Tuyl (Theorem 4.1  \cite{Adam}) all algorithms outlined in this paper provide us with $f$-ideals in degrees $n-d_{i}$ via the Newton complementary dual, where the $d_{i}$ are the degrees of the original $f$-ideal and $n$ is the number of variables in the ring.

\begin{acknowledgement}
The authors would like to thank Adam Van Tuyl who made contributions to this paper in its preliminary stage, and with whom we had many useful discussions. The first author is partially supported by Louisiana Board of Regents (grant \# LEQSF(2017-19)-ENH-TR-25). The third author is supported by the Higher Education Commission of Pakistan (No: $7515$$\slash$Punjab$\slash$NRPU$\slash$R$\&$D$\slash$HEC$\slash$2017). Part of this work was done while the third author visited Tulane University. The authors would like to thank Tulane University for its hospitality.
\end{acknowledgement}

\section{Preliminaries} \label{sec.prel}

In this section, we introduce important notation and terminology that will be used in the paper. For a more extensive exposition we direct the reader to \cite{StanleyBook}.


Throughout the paper, $\Delta$ denotes a \emph{simplicial complex} over the vertex set $V = \{x_1, \dots, x_n\}$. That is, $\Delta$ consists of nonempty subsets of $V$ with the following properties:
\begin{enumerate}
	\item[$(i)$] $\{x_i\} \in \Delta$ for all $i = 1, \dots, n$, and
	\item[$(ii)$] if $F \in \Delta$ and $G \subseteq F$ then $G \in \Delta$.
\end{enumerate}
Elements in $\Delta$ are called \emph{faces} of $\Delta$. Maximal faces of $\Delta$ are called \emph{facets}. The set of facets of $\Delta$ is denoted by $\F(\Delta)$. Obviously, $\Delta$ is completely determined by its facets.

For a face $F \in \Delta$, its \emph{dimension} is $\dim F := \#F - 1$, where $\#F$ denotes the cardinality of $F$. The \emph{dimension} of $\Delta$ is given by
$$\dim \Delta = \max \{\dim F ~\big|~ F \in \Delta\}.$$


Let $\kk$ be a field and identify the vertices in $V$ with the variables in the polynomial ring $R = \kk[x_1, \dots, x_n]$. The following constructions, see \cite{StanleyBook, faridi}, give one-to-one correspondences between simplicial complexes over $V$ and squarefree monomial ideals (generated in degrees at least 2) in $R$.

\begin{definition}
	Let $\Delta$ be a simplicial complex over $V = \{x_1, \dots, x_n\}$.
	\begin{enumerate}
		\item The \emph{Stanley-Reisner ideal} of $\Delta$ is defined to be
		$$I_\Delta = \langle x_{i_1} \dots x_{i_r} ~\big|~ \{x_{i_1}, \dots, x_{i_r}\} \not\in \Delta\rangle \subseteq R.$$
		\item The \emph{facet ideal} of $\Delta$ is defined to be
		$$I(\Delta) = \langle x_{j_1} \dots x_{j_s} ~\big|~ \{x_{j_1}, \dots, x_{j_s}\}\in\F(\Delta) \rangle \subseteq R.$$
	\end{enumerate}
\end{definition}

For a squarefree monomial ideal $I \subseteq R$ (generated in degrees at least 2), the dual constructions give us the \emph{Stanley-Reisner complex} $\Delta_I$ and the \emph{facet complex} $\Delta(I)$ associated to $I$. It is easy to see that
$$I_{\Delta_I} = I = I(\Delta(I)) \text{ and } \Delta_{I_\Delta} = \Delta = \Delta(I(\Delta)).$$


The main objects of our study are \emph{$f$-ideals}, which are characterized by the \emph{$f$-vectors} of corresponding simplicial complexes. We now recall the definition of $f$-vectors of simplicial complexes and $f$-ideals.

\begin{definition}
	Let $\Delta$ be a simplicial complex of dimension $d$. The \emph{$f$-vector} of $\Delta$ is a sequence of integers $f(\Delta) = (f_{-1}, f_0, \dots, f_d)$, where $f_i$ is the number of $i$-dimensional faces in $\Delta$ for any $i$ (with the convention that $f_{-1} = 1$ unless $\Delta$ is the empty complex).
\end{definition}


\begin{definition}
	A squarefree monomial ideal $I \subseteq R$ generated in degrees at least 2 is called an \emph{$f$-ideal} if its Stanley-Reisner complex $\Delta_I$ and its facet complex $\Delta(I)$ have the same $f$-vectors. i.e. $f(\Delta_I)=f(\Delta(I))$.
\end{definition}

An $f$-ideal $I$ is called \emph{pure} if it is \emph{equi-generated}; that is, its minimal generators are all of the same degree. Equivalently, an $f$-ideal $I$ is pure if and only if its Stanley-Reisner and facet complexes are pure simplicial complexes. An $f$-ideal that is not pure is called a \emph{mixed} $f$-ideal.

Pure $f$-ideals can be characterized by the following property of being a \emph{perfect} set of the collection of minimal generators.


\begin{definition} Let $\M_{n,d}$ denote the collection of squarefree monomials of degree $d$ in $R = \kk[x_1, \dots, x_n]$.
	\begin{enumerate}
		\item A subset $\A \subseteq \M_{n,d}$ is called \emph{upper perfect} if
		$$\left\{fx_i ~\big|~ f \in \A, x_i \nmid f, 1 \le i \le n\right\} = \M_{n,d+1}.$$
		\item A subset $\A \subseteq \M_{n,d}$ is called \emph{lower perfect} if
		$$\left\{ \dfrac{g}{x_i} ~\Big|~ g \in \A, x_i \mid g, 1 \le i \le n\right\} = \M_{n,d-1}.$$
		\item A subset $\A \subseteq \M_{n,d}$ is called \emph{perfect} if it is both upper perfect and lower perfect.
	\end{enumerate}
\end{definition}

\begin{theorem}[\protect{\cite[Theorem 2.3]{tswuPubl}}]
	\label{thm.PerfectCharacterization}
	Let $I \subseteq R$ be a squarefree monomial ideal generated by $\frac{1}{2}{n \choose d}$ generators of degree $d \ge 2$. Let $G(I)$ be the set of minimal generators of $I$. Then $I$ is a pure $f$-ideal if and only if $G(I)$ is a perfect set.
\end{theorem}

We shall let $\LP(n,d)$ and $\UP(n,d)$ be the collections of lower perfect and upper perfect sets in $\M_{n,d}$, respectively. Furthermore, let $\I(n,d)$ be the collection of squarefree monomial ideals in $R$ generated in degree $d$ with exactly $\frac{1}{2}{n \choose d}$ generators, and let $\fideal(n,d)$ be the collection of pure $f$-ideals in $R$ generated in degree $d$. We follow the convention of using a subscript $d$, e.g, $I_d$, to denote the $d$-th graded component of an ideal $I$. \\


A useful framework for testing whether an arbitrary squarefree monomial ideal is an $f$-ideal was formulated by Budd and Van Tuyl \cite{Adam} which we quickly recall for later use.

\begin{notation}[\protect{See \cite{Adam}}]\label{abcd}
For a squarefree monomial $I\subset R$ we define the following partition of squarefree monomials in a given degree $d$:
\begin{align*}
A_{d}(I)&=\{M\in \M_{n,d} ~\big|~ M\notin I_{d}\text{ and $M$ does not divide any element of $G(I)$}\} \\
B_{d}(I)&=\{M\in \M_{n,d} ~\big|~ M\notin I_{d} \text{ and $M$ divides some element of $G(I)$}\}\\
C_{d}(I)&=\{M\in \M_{n,d} ~\big|~  M\in G(I)\} \\
D_{d}(I)&=\{M\in \M_{n,d} ~\big|~ M\in I_{d}\backslash G(I)\}
\end{align*}
\end{notation}
Occasionally, we may refer to elements of type $A$ as \emph{nonfaces} and those of type $C$ as \emph{generators} of $I$ for obvious reasons.

\begin{lemma}[\protect{\cite[Lemma 2.5]{Adam}}]\label{abcd.lem}
Let $I\subseteq \kk[x_{1},\dots ,x_{n}]$ be a squarefree monomial ideal. Then $I$ is an f-ideal if and only if $\#A_{d}(I)=\#C_{d}(I)$ for all $0\le d\le n$.
\end{lemma}


\section{Density of $f$-ideals} \label{sec.density}

In this section, we examine the density of pure $f$-ideals among ideals in $n$ variables, equi-generated in degree $d$ with $\frac{1}{2}{n \choose d}$ generators. For simplicity of notation, throughout the section, we shall let $m = \frac{1}{2}{n \choose d}$. Our first main result reads as follows.

\begin{theorem} \label{thm.density}
	For a fixed integer $d \ge 2$, we have
	$$\lim_{n \to \infty} \dfrac{\#\fideal(n,d)}{\#\I(n,d)} = 0.$$
\end{theorem}

In light of Theorem \ref{thm.PerfectCharacterization}, Theorem \ref{thm.density} is an immediate consequence of the following result on the asymptotic density of lower perfect and upper perfect sets.

\begin{theorem} \label{thm.ULP}
For a fixed integer $d \ge 2$, we have
$$\lim_{n \to \infty} \dfrac{\#\left(\LP(n,d)\cap\I(n,d)\right)}{\#\I(n,d)} = 1 \quad \text{ and } \quad \lim_{n \to \infty} \dfrac{\#\left(\UP(n,d)\cap\I(n,d)\right)}{\#\I(n,d)} = 0.$$
\end{theorem}

\begin{proof} For a monomial $g \in \M_{n,d-1}$, let $L_g \subset \I(n,d)$ denote the set of ideals which fail to be lower perfect because $g$ divides no generator of the ideal. In other words, $$L_g := \{I \in \I(n,d) ~\big|~ g\nmid f \text{ for all } f \in \G(I)\}.$$
Then $\LP(n,d)\cap\I(n,d) = \I(n,d) \setminus \bigcup_{g \in \M_{n,d-1}} L_g$, and

\begin{equation} \notag
\begin{aligned}
\#\left(\LP(n,d)\cap\I(n,d)\right) & = & & \#\I(n,d) - \#\bigcup_{g \in \M_{n,d-1}} L_g \\
& \geq & & \#\I(n,d) - \sum_{g \in \M_{n,d-1}} \#L_g.
\end{aligned}
\end{equation}

It can be seen that an ideal is in $L_g$ if and only if it does include any of the monomials $x_j g$ for $j \not\in \supp(g)$ as generators. Equivalently, its $m$ generators are chosen from among the $2m - (n - (d - 1))$ possibilities not of the form $x_j g$. This implies that $\# L_g = {2m - n + d - 1 \choose m}$. Moreover, $\#\M_{n,d-1} = {n \choose d-1}$, so we have $$\sum_{g \in \M_{n,d-1}} \#L_g = {n \choose d-1}{2m - n + d - 1 \choose m}.$$
Therefore,
\begin{equation} \notag
\begin{aligned}
\lim_{n \to \infty} \dfrac{\#\left(\LP(n,d)\cap\I(n,d)\right)}{\#\I(n,d)}
& = & & 1 - \lim_{n \to \infty} \dfrac{{n \choose d-1} \cdot {2m - n + d - 1 \choose m}}{{2m \choose m}} \\
& = & & 1 - \lim_{n \to \infty} {n \choose d-1} \prod_{i=0}^{n-d} \dfrac{m-i}{2m-i} \\
& \ge & & 1 - \lim_{n \to \infty} {n \choose d-1} \left( \dfrac{m}{2m} \right)^{n - d + 1} \\
& = & & 1 - \lim_{n \to \infty} {n \choose d-1} \dfrac{1}{2^{n-d+1}} \\
& = & & 1.
\end{aligned}
\end{equation}
Here, the last equality follows from the fact that ${n \choose d-1} \le n^{d-1}$, which grows much slower than $2^{n-d+1}$ as $n \rightarrow \infty$. This establishes the first desired limit.

To prove the second limit, fix arbitrary constants $0 < \epsilon < 1$ and $0 < \theta < \dfrac{1}{2}$. Clearly, ${\displaystyle \lim_{n \rightarrow \infty}\dfrac{m-n}{2m-n} = \dfrac{1}{2}}$ (recall $m=\frac{1}{2}\binom{n}{d}$). Thus, $\dfrac{m-n}{2m-n} > \theta$ for $n \gg 0$. Furthermore, if we write $n = (d+1)l + r$, for $0 \le r < d+1$, then for $n \gg 0$ we also have
$$l > \dfrac{\ln \epsilon}{\ln(1-\theta^{d+1})}, \text{ i.e., } \left(1-\theta^{d+1}\right)^l < \epsilon.$$

Let \[\mathcal{U}=\{u_{1}=x_{1}x_{2}\dots x_{d+1}, u_{2}=x_{d+2}x_{d+3}\dots x_{2d+2}, \dots, u_{l}=x_{(l-1)(d+1)+1}\dots x_{l(d+1)}\}\] \[\mathcal{B}_{p}=\left\{f_{p,q+1} = \dfrac{u_p}{x_{p(d+1)-q}}\;\middle|\;0 \le q \le d\right\} \:\:\: \text{for  $1 \le p \le l$} \] Consider an arbitrary subset $\A \subset \M_{n,d}$ consisting of $m$ monomials. We say that $\A$ \emph{covers} $u_p$ if $\A\cap\mathcal{B}_{p}\ne\emptyset$. We say that $\A$ \emph{covers} $\mathcal{U}$ if $\A\cap\mathcal{B}_{p}\ne\emptyset$ for $1\le p\le l$. It is clear to see that if $\A$ is an upper perfect set then $\A$ must cover $\mathcal{U}$. (However, this is not an if and only if statement.)

We would like to find an upper bound on the probability that $\A$ covers $\mathcal{U}$ and show that this goes to zero as $n \rightarrow \infty$. We do this by finding an upper bound on the probability $\A$ covers $u_{p}$. Note that the probabilities of the $u_{p}$ being covered by $\A$ are not independent of each other. For a single $f_{p,q+1}$ with no other assumptions the probability it belongs to $\A$ is $\frac{\# \A}{\#\M_{n,d}}$. If we know whether or not $f_{1,1}$ is in $\A$ then the probability that $f_{p,q+1}$  belongs to $\A$ is $\frac{\#\A -1}{\#\M_{n,d}-1}$ if $f_{1,1}\in\A$ or $\frac{\#\A}{\#\M_{n,d}-1}$ if $f_{1,1}\notin\A$. We see then that an upper bound will be given by the case were we assume all other factors of all of the $u_{i}$ are not in $\A$. This gives $\frac{m}{2m-((d+1)l-1)}$ which we can for convenience bound above by $\frac{m}{2m-n}$.

Thus, the probability that $f_{p,q+1}\notin\A$ is bounded below in all cases (that is regardless of whether any other such factors belong to $\A$) by $1 - \dfrac{m}{2m-n}$. It follows that the probability that $\A$ covers $u_p$ is bounded above by
$$1 - \left(1 - \dfrac{m}{2m-n}\right)^{d+1} = 1 - \left(\dfrac{m-n}{2m-n}\right)^{d+1}.$$

In the end, we deduce that the probability that $\A$ covers $u_1, \dots, u_l$ is bounded above by
$$\left[1 - \left(\dfrac{m-n}{2m-n}\right)^{d+1}\right]^l < \left(1 - \theta^{d+1}\right)^l < \epsilon.$$
This inequality is true for any $0 < \epsilon < 1$. The assertion for the asymptotic density of upper perfect sets then follows.
\end{proof}

Theorem \ref{thm.density} verifies the expected behavior that $f$-ideals are rare. It is naturally desirable to construct explicit $f$-ideals. This is the content of our next section.


\section{Mixed $f$-ideals in degrees 2 and 3} \label{sec.mixed}

In this section, we provide algorithms to construct mixed $f$-ideals generated in degrees 2 and 3. Particularly, it is shown that mixed $f$-ideals generated in degrees 2 and 3 always exist if the number of variables $n$ is at least 6 when $n$ is even and when $n$ is at least 31 in the case that $n$ is odd (in contrast the case of pure f-ideals in degree $d$ where $\binom{n}{d}$ must be even by Theorem \ref{thm.PerfectCharacterization}).

\subsection{Odd number of variables.} We shall first consider the case when the number of variables, n, is odd and at least 7. We provide two algorithm's to construct $f$-ideals, one for cases where $n\equiv 1,7$ (mod 8) and the other for $n \equiv 3,5$ (mod 8).
\begin{algorithm}
	\label{23alg.odd1}
	Let $n = 2k+1$ for $k \ge 3$, $k\equiv 0\ (\text{mod }4)$ or $k\equiv 3\ (\text{mod } 4)$. Let $R = \kk[x_1, \dots, x_n]$. We construct an ideal $I\subset R$ as follows:
	\begin{enumerate}[align=left]
		\item[\textsc{Step 1:}] Set $G_1 = \{x_{2i}x_{2j} ~\big|~ 1 \le i < j \le k\}$.
		\item[\textsc{Step 2:}] Set $\text{NF} = \{x_ix_{j} ~\big|~ i \not\in 2\ZZ, j \in 2\ZZ, \text{ and } i+1 < j\}$.
		\item[\textsc{Step 3:}] Set $G_2 = \{x_ix_jx_{2k+1} ~\big|~ x_ix_j \not\in (G_1 \cup \text{NF})\}$.
		\item[\textsc{Step 4:}] Set $G_3 = \{x_ix_jx_l ~\big|~ 1 \le i < j < l\le k, \text{ and } i,j, l \not\in 2\ZZ\}$.
		\item[\textsc{Step 5:}] Select
		$$\dfrac{1}{2} \cdot \dfrac{k(k+1)(4k-1)}{6} - \dfrac{k(k-1)(k-2) + 6k^2}{6}$$
		degree 3 monomials in the complement of $G_2 \cup G_3$, which are not divisible by monomials of $G_1$ or $\text{NF}_1$. Denote the set of these monomials by $G_4$.
		\item[\textsc{Step 6:}] Output $I = \langle G_1 \cup G_2 \cup G_3 \cup G_4\rangle$, the monomial ideal generated by monomials in $G_1, \dots, G_4$.
	\end{enumerate}
\end{algorithm}

\begin{theorem}
	\label{thm.deg23odd1}
	The ideal $I$ output from Algorithm \ref{23alg.odd1} is an $f$-ideal.
\end{theorem}

\begin{proof} Our goal is to apply Lemma \ref{abcd.lem}. Let $A_{i}$, $B_{i}$, $C_{i}$ and $D_{i}$ be as in Notation \ref{abcd}. $\#A_1 = \#C_1 = 0$ because every variable shows up in some generator. We proceed to show that $\#A_d = \#C_d$ for all $2 \le d \le n$.\\
	
It is easy to see that both $G_1$ and $\text{NF}$ contain exactly $\binom{k}{2}$ elements. This guarantees that $\#A_2 = \#C_2 = \binom{k}{2}$ since all other degree 2 squarefree monomials divide some element of $G_{2}$ and so belong to $B_{2}$. \\

Consider elements of $\mathcal{M}_{2k+1,3}$ that are not in $D_3$. From the construction, these monomials are exactly the products of 2 variables of odd indices and one with even index and the product of 3 variables of odd indices. Thus,
$$\#(\mathcal{M}_{2k+1,d}\backslash D_{3}) = \binom{k+1}{2}\binom{k}{1} + \binom{k+1}{3} = \dfrac{k(k+1)(4k-1)}{6}.$$
By our assumption that $k\equiv 0\;\text{mod}(4)$ or $k\equiv 3\;\text{mod}(4)$ this is always even and so it is possible that the remaining monomials can be split evenly between $A_{3}$ and $C_{3}$.

 We note that $\#G_{2}=\binom{2k}{2}-2\binom{k}{2}=k^{2}$ and $\#G_{3}=\binom{k}{3}$, so we have already selected $\dfrac{k(k-1)(k-2)+6k^{2}}{6}$ elements to be in $C_{3}$. We note that this is always less than half of $\#(\mathcal{M}_{2k+1,d}\backslash D_{3})$ when $k\ge 3$. Hence we must add monomials into $C_{3}$ until \[\#C_{3}=\dfrac{1}{2}\dfrac{k(k+1)(4k-1)}{6}\] at which point the remaining monomials will be in $A_{3}$ since there are no degree 4 generators.

 It is important to note that picking out $G_4$ in Step 5 is possible. To see this, it suffices to show that there are at least $\dfrac{1}{2}\dfrac{k(k+1)(4k-1)}{6}-\dfrac{k(k-1)(k-2)+6k^{2}}{6}$ squarefree monomials of degree 3 which are not in $D_{3}\cup G_{2}\cup G_{3}$ and not divisible by elements in $\text{NF}_{1}$. Indeed, the monomials $x_{a}x_{b}x_{c}$ where $a$ and $b$ are odd (excluding $x_{2n+1}$) and $c$ even with $c<\text{min}(a,b)+1$ are not multiples of elements of $\text{NF}$ nor do they belong to $D_{3}$ or $G_{2}\cup G_{3}$. We can count these monomials by fixing the smaller of the two odd numbers and then cycling through all larger odd numbers and all smaller even numbers in the following way:
 \begin{enumerate}
 	\item[1)] For $x_{1}$ there are no options of this type (since there are no smaller even numbers).
 	\item[2)] For $x_{3}$ we can pick $x_{2}$ or $x_{4}$ and any of the $k-2$ higher odds.\\
 	\vdots
 	\item[k-1)] For $x_{2k-3}$ we can pick $x_{2}$ through $x_{2k-2}$ and have higher odd $x_{2k-1}$.
\end{enumerate}
The total count from this is
\[2(k-2)+3(k-3)+4(k-4)+\cdots + (k-2)2+(k-1)=\frac{(k-1)k(k+1)}{6}-(k-1).\]
Observe that
 \[\dfrac{(k-1)k(k+1)}{6}-(k-1)\ge\dfrac{1}{2}\dfrac{k(k+1)(4k-1)}{6}-\dfrac{k(k-1)(k-2)+6k^{2}}{6}.\]
Therefore, it follows that there are enough monomials from which we can select  $G_{4}$ to force $\#C_{3}=\#A_{3}$.

Finally, it can be seen that all squarefree monomials of degree $d \ge 4$ are in $D_d$, i.e., $\#A_d = \#C_d = 0$, since they are divided by either elements in $G_{1}$, $G_{2}$ or $G_{3}$. The conclusion now follows from Lemma \ref{abcd.lem}.
\end{proof}

The following example illustrates how Algorithm \ref{23alg.odd1} works.

\begin{example} Let $n = 7$ (and $k = 3$). Consider the monomial ideal $I = \langle G_1 \cup G_2 \cup G_3 \cup G_4\rangle$, where $G_1, \dots, G_4$ are described as follows:
	\begin{itemize}
		\item $G_1 = \{x_{2}x_{4},x_{2}x_{6},x_{4}x_{6}\}$.
		\item $\text{NF} = \{x_{1}x_{4},x_{1}x_{6},x_{3}x_{6}\}$.
		\item $G_2 = \{x_{1}x_{2}x_{7}, x_{3}x_{4}x_{7}, x_{5}x_{6}x_{7}, x_{1}x_{3}x_{7}, x_{1}x_{5}x_{7}, x_{3}x_{5}x_{7}, x_{2}x_{5}x_{7}, x_{4}x_{5}x_{7}\}$.
		\item $G_3 = \{x_{1}x_{3}x_{5}\}$.
		\item $G_4 = \{x_{2}x_{3}x_{5}\}$.
	\end{itemize}
	\vspace{-.5cm}
	 \begin{align*}
	I = \langle&x_{2}x_{4},x_{2}x_{6},x_{4}x_{6},x_{1}x_{2}x_{7},x_{3}x_{4}x_{7}, x_{5}x_{6}x_{7},x_{1}x_{3}x_{7},\\
	&x_{1}x_{5}x_{7},x_{3}x_{5}x_{7},x_{2}x_{5}x_{7},x_{4}x_{5}x_{7},x_{1}x_{3}x_{5},x_{2}x_{3}x_{5}\rangle
\end{align*}
\end{example}

We now present the algorithm covering the cases where $n=2k+1$ and $k\equiv 1 \ (\text{mod }4)$ or $k\equiv 2 \ (\text{mod } 4)$.
\begin{algorithm}
	\label{23alg.odd2}
	Let $n = 2k+1$ for $k \ge 15$, $k\equiv 1 \ (\text{mod }4)$ or $k\equiv 2 \ (\text{mod } 4)$. Let $R = \kk[x_1, \dots, x_n]$. We construct an ideal $I\subset R$ as follows:
	\begin{enumerate}[align=left]
		\item[\textsc{Step 1:}] Set $G_1 = \{x_{2i}x_{2j} ~\big|~ 1 \le i < j \le k\}\cup \{x_{3}x_{2},x_{5}x_{2}\}$.
		\item[\textsc{Step 2:}] Set $\text{NF} = \{x_ix_{j} ~\big|~ i \not\in 2\ZZ, j \in 2\ZZ, \text{ and } i+1 < j\}\cup\{x_{2}x_{7},x_{2}x_{9}\}$.
		\item[\textsc{Step 3:}] Set $G_2 = \{x_ix_jx_{2k+1} ~\big|~ x_ix_j \not\in (G_1 \cup \text{NF})\}$.
		\item[\textsc{Step 4:}] Set $G_3 = \{x_ix_jx_l ~\big|~ 1 \le i < j < l\le k, \text{ and } i,j, l \not\in 2\ZZ\}$.
		\item[\textsc{Step 5:}] Select
		$$\dfrac{1}{2} \cdot \dfrac{4k^{3}+3k^{2}-13k+6}{6} - \dfrac{k(k-1)(k-2) + 6k^2-24}{6}$$
		degree 3 monomials in the complement of $G_2 \cup G_3$, which are not divisible by monomials of $G_1$ or $\text{NF}$. Denote the set of these monomials by $G_4$.
		\item[\textsc{Step 6:}] Output $I = \langle G_1 \cup G_2 \cup G_3 \cup G_4\rangle$, the monomial ideal generated by monomials in $G_1, \dots, G_4$.
	\end{enumerate}
\end{algorithm}

\begin{theorem}
	\label{thm.deg23odd2}
	The ideal $I$ output from Algorithm \ref{23alg.odd2} is an $f$-ideals.
\end{theorem}

\begin{proof} We follow the same line of arguments as that of \ref{thm.deg23odd1}.
It is easy to see that both $G_1$ and $\text{NF}$ contain exactly $\binom{k}{2}+2$ elements. This guarantees that $\#A_2 = \#C_2 = \binom{k}{2}+2$ since all other degree 2 squarefree monomials divide some element of $G_{2}$ and so belong to $B_{2}$.

Consider elements of $\mathcal{M}_{2k+1,3}$ that are not in $D_3$. As in Theorem \ref{thm.deg23odd1}, these monomials are exactly the products of 2 variables of odd indices and one with even index and the product of 3 variables of odd indices, with the additional elements which are divided by $x_{2}x_{3}$ and $x_{2}x_{5}$ which divide $k$ and $k-1$ additional elements (these are of the form $x_{2}x_{3}x_{\alpha}$ and $x_{2}x_{5}x_{\beta}$ where $\alpha$ and $\beta$ are odd). Thus,
$$\#(\mathcal{M}_{2k+1,d}\backslash D_{3}) = \binom{k+1}{2}\binom{k}{1} + \binom{k+1}{3} -(2k-1) = \dfrac{4k^{3}+3k^{2}-13k+6}{6}.$$
It follows from the assumption on $k$ that $\#(\M_{2k+1,d} \setminus D_3)$ is always even and so it is possible that the remaining monomials can be split evenly between $A_{3}$ and $C_{3}$.

 Note that $\#G_{2}=\binom{2k}{2}-2\binom{k}{2}-4=k^{2}-4$ and $\#G_{3}=\binom{k}{3}$, so we have already selected $\dfrac{k(k-1)(k-2)+6k^{2}-24}{6}$ elements to be in $C_{3}$. Note also that this is always less than half of $\#(\mathcal{M}_{2k+1,d}\backslash D_{3})$ when $k\ge 15$. Thus, we add monomials into $C_{3}$ until \[\#C_{3}=\dfrac{1}{2}\dfrac{k(k+1)(4k-1)}{6},\]
 at which point the remaining monomials will be in $A_{3}$ since there are no degree 4 generators.

 Again, we have to show that the choice for $G_4$ in Step 5 is possible. That is, there are at least $\dfrac{1}{2}\dfrac{4k^{3}+3k^{2}-13k+6}{6}-\dfrac{k(k-1)(k-2)+6k^{2}-24}{6}$ squarefree monomials of degree 3 which are not in $D_{3}\cup G_{2}\cup G_{3}$ and not divisible by elements in $\text{NF}$. The monomials $x_{a}x_{b}x_{c}$ where $a$ and $b$ are odd (excluding $x_{2n+1}$) and $c$ even with $c<\text{min}(a,b)+1$ are not multiples of elements of $\text{NF}$ nor do they belong to $D_{3}$ or $G_{2}\cup G_{3}$, there is now however an exception since we must avoid multiples of $x_{2}x_{3},x_{2}x_{5},x_{2}x_{7},x_{2}x_{9}$ so at the end we will subtract by $4k-6$. As before, we can count these by fixing the smaller of the two odd numbers and then cycling through all larger odd numbers and all smaller even numbers in the following way:
\begin{enumerate}
	\item[1)] For $x_{1}$ there are no options of this type (since there are no smaller even numbers).
	\item[2)] For $x_{3}$ we can pick $x_{2}$ or $x_{4}$ and any of the $k-2$ higher odds.\\
	\vdots
	\item[k-1)] For $x_{2k-3}$ we can pick $x_{2}$ through $x_{2k-2}$ and have higher odd $x_{2k-1}$.
\end{enumerate}
After subtracting the $4k-6$ multiples of $x_{2}x_{3},x_{2}x_{5},x_{2}x_{7},x_{2},x_{9}$, the total number we obtain is
\[2(k-2)+3(k-3)+4(k-4)+\cdots (k-2)2+(k-1)-(4k-6)=\dfrac{(k-1)k(k+1)}{6}-(k-1)-(4k-6)\]
Observe that, for $k \ge 15$,
 \[\dfrac{(k-1)k(k+1)}{6}-(k-1)-(4k-6)\ge\dfrac{1}{2}\dfrac{4k^{3}+3k^{2}-13k+6}{6}-\dfrac{k(k-1)(k-2)+6k^{2}-24}{6}.\]
Therefore, it follows that there are enough monomials from which we can select  $G_{4}$ to force $\#C_{3}=\#A_{3}$.

Finally, it can be seen that all squarefree monomials of degree $d \ge 4$ are in $D_d$, i.e., $\#A_d = \#C_d = 0$, since they are divided by either elements in $G_{1}$, $G_{2}$ or $G_{3}$. The conclusion now follows from Lemma \ref{abcd.lem}.
\end{proof}


\subsection{Even number of variables.} We shall now consider the case when the number of variables is even. Our algorithm reads as follows.

\begin{algorithm}
	\label{23alg.even}
	Let $n = 2k$ for $k\ge 4$ and let $R = \kk[x_1, \dots, x_n]$.
\begin{enumerate}[align=left]
\item[\textsc{Step 1:}] Set $G_{1}=\{x_{2i}x_{2j} ~\big|~ 1 \le i < j < k\}$.
\item[\textsc{Step 2:}] Set $\text{NF}=\{x_{i}x_{j} ~\big|~ i\notin 2\mathbb{Z}, j\in 2\mathbb{Z}, \text{ and } i+1 < j < 2k\}$.
\item[\textsc{Step 3:}] Set $G_{2}=\{x_{i}x_{j}x_{2k-1} ~\big|~ i < j < 2k-1, x_{i}x_{j}\notin (\text{NF}\cup G_{1})\}$.
\item[\textsc{Step 4:}] Set $G_{3}=\{x_{i}x_{j}x_{l} ~\big|~ 1 \le i<j<l < 2k-1 \text{ and } i,j,l\notin 2\mathbb{Z}\}$.
\item[\textsc{Step 5:}] Set $G_{4}=\{x_{2k}x_{i}x_{j} ~\big|~ 1\le i < j \le 2k-1 \text{ and } i,j\notin 2\mathbb{Z}\}$
\item[\textsc{Step 6:}] Set $G_{5}=\{x_{2k}x_{2k-1}x_{i} ~\big|~ i\in 2\mathbb{Z}\}$
\item[\textsc{Step 7:}] Select
$$\frac{k^{3}-3k^{2}-4k+6}{6}$$
degree 3 monomials from the complement of $G_{2}\cup G_{3}\cup G_{4}\cup G_{5}$, which are not divisible by monomials of $G_{1}$ or $\text{NF}$. Denote the set of these monomials by $G_{6}$.
\item[\textsc{Step 6:}] Output $I=\langle G_{1}\cup\cdots\cup G_{6}\rangle$, the monomial ideals generated by monomials in $G_1, \dots, G_6$.
\end{enumerate}
\end{algorithm}

\begin{theorem}
	\label{thm.deg23even}
	The ideal $I$ output from Algorithm \ref{23alg.even} is an $f$-ideal.
\end{theorem}

\begin{proof} Once again, we use the same line of arguments as that of Theorem \ref{thm.deg23odd1}. To start, we have $\#A_{2}(I)=\#C_{2}(I)=\binom{k-1}{2}$ by construction and that all other squarefree monomials in degree 2 are in $B_{2}$ since they divide elements in $G_{2}$, $G_{4}$ or $G_{5}$.

Consider elements of $\mathcal{M}_{2k,3}$ that are not in $D_3$. These monomials are exactly the products of 2 variables of odd indices and one with even index and the product of 3 variables of odd indices and squarefree monomials of the form $x_{2k}x_{a}x_{b}$ where $a|2$ and $b\nmid 2$. Thus, $$\#(\mathcal{M}_{2k,d}\backslash D_{3}) = \binom{k}{2}\binom{k}{1} + \binom{k}{3} + (k-1)k = \dfrac{4k(k-1)(k+1)}{6}.$$
This is always even and so it is possible that the remaining monomials can be split evenly between $A_{3}$ and $C_{3}$.

Note that $\#G_{2}=\binom{2k-2}{2}-2\binom{k-1}{2}=(k-1)^{2}$ and $\#G_{3}=\binom{k-1}{3}$,
$\#G_{4}=\binom{k}{2}$ and $\#G_{5}=k-1$.

 so we have already selected $\dfrac{(k-1)(k^{2}+4k+6)}{6}$ elements to be in $C_{3}$. Note further that this is always less than half of $\#(\mathcal{M}_{2k+1,d}\backslash D_{3})$ when $k\ge 6$. Hence we add monomials into $C_{3}$ until \[\#C_{3}=\dfrac{1}{2}\dfrac{4k(k-1)(k+1)}{6},\]
at which point the remaining monomials will be in $A_{3}$ since there are no degree 4 generators. \\

 Again, we need to show that it is possible to pick out $G_6$. That is, there are at least \[\dfrac{1}{2}\dfrac{4k(k-1)(k+1)}{6}-\dfrac{(k-1)(k^{2}+4k+6)}{6}=\frac{k^{3}-3k^{2}-4k+6}{6}\] squarefree monomials of degree 3 which are not in $D_{3}\cup G_{2}\cup G_{3}$ and are not divisible by elements in $\text{NF}$. The monomials $x_{a}x_{b}x_{c}$ where $a$ and $b$ are odd (excluding $x_{2n-1}$) and $c$ even with $c\le\text{min}(a,b)+1$ are not multiples of elements of $\text{NF}$ nor do they belong to $D_{3}$ or $G_{2}\cup G_{3}\cup G_{4}\cup G_{5}$. As before, we can count these by fixing the smaller of the two odd numbers and then cycling through all larger odd numbers and all smaller even numbers:
\begin{enumerate}
	\item[1)] For $x_{1}$ we can pick $x_{2}$ and any of the $k-2$ higher odds (excluding $x_{2k-1}$).
	\item[2)] For $x_{3}$ we can pick $x_{2}$ or $x_{4}$ and any of the $k-3$ higher odds.
	\item[3)] For $x_{5}$ we can pick $x_{2}$ or $x_{4}$ or $x_{6}$ with any of the $k-4$ higher odds.\\
	\vdots
	\item[k-2)] For $x_{2k-5}$ we can pick $x_{2}$ through $x_{2k-4}$ and have higher odd $x_{2k-3}$.
\end{enumerate}

The total count is
\[(k-2)+2(k-3)+3(k-4)+\cdots (k-3)2+(k-2)=\sum_{i=1}^{k-2}i(k-1-i)=\frac{k^{3}-3k^{2}+2k}{6}\]
Observe that
\[\frac{k^{3}-3k^{2}+2k}{6}\ge \frac{k^{3}-3k^{2}-4k+6}{6}\]
Therefore, there are enough monomials which are not multiples of our previously selected degree 2 nonfaces  which we can select as generators to reach
\[\frac{1}{2}\frac{4k(k-1)(k+1)}{6}\]
which was half of $\#(\M_{2k,3}\backslash D_{3})$. It follows that the remaining degree three monomials belong to $A_{3}$.

Finally, it can be seen that all squarefree monomials of degree $d \ge 4$ are in $D_d$, i.e., $\#A_d = \#C_d = 0$, since they are divided by either elements in $G_{1}$, $G_{2}$, $G_{3}$ or $G_{4}$. The conclusion now follows from Lemma \ref{abcd.lem}.
\end{proof}


\subsection{Additional algorithms.} We end this section by providing two more algorithms to construct mixed $f$-ideals generated in degrees 2 and 3 for the cases where $n\equiv 3$ (mod 4) and $n\equiv 2$ (mod 4) respectively.


\begin{algorithm}
	\label{alg.odd3} Let $k\geq3$ be an odd positive integer and let $l$ and $m$ be two fixed distinct positive integers less than or equal to $k$. Let $R = \kk[x_1,x_2,\ldots,x_{k},y_1,y_2,\dots,y_{k+1}]$.
	\begin{enumerate}[align=left]
		\item[\textsc{Step 1:}] Set $G_1=\{x_{i}x_j ~\big|~ i\neq j\in[k]\}\backslash\{x_lx_m\}$.
		\item[\textsc{Step 2:}] Set $G_2=\{y_{i}y_j ~\big|~ i\neq j\in[k+1]\}$.
		\item[\textsc{Step 3:}] Set $G_3$ to be any subset of the set $M=\{y_{i}x_lx_m ~\big|~ i\in[k+1]\}$ such that  $\#G_3=\frac{k+1}{2}$.
		\item[\textsc{Step 4:}] Output $I = \langle G_1 \cup G_2 \cup G_3\rangle$, the monomial ideal generated by monomials in $G_1, G_2 \ \text{and}\  G_3$.
	\end{enumerate}
\end{algorithm}


\begin{algorithm}
	\label{alg.odd2} Let $k\geq4$ be an even positive integer and let $l$ and $m$ be two fixed distinct positive integers less than or equal to $k+2$. Let $R = \kk[x_1,x_2,\ldots,x_k,y_1,y_2,\dots,y_{k+2}]$.
	\begin{enumerate}[align=left]
		\item[\textsc{Step 1:}] Set $G_1 =\{x_{i}x_{j} ~\big|~ i\neq j\in[k]\} $.
		\item[\textsc{Step 2:}] Set $G_2=\{y_{i}y_j ~\big|~ i\neq j\in[k+2]\}\backslash\{y_ly_m\}$
		\item[\textsc{Step 3:}] Set $G_3$ to be any subset of the set $M=\{x_{i}y_ly_m ~\big|~ i\in[k]\}$ such that  $\#G_3=\frac{k}{2}$.
		\item[\textsc{Step 4:}] Output $I = \langle G_1 \cup G_2 \cup G_3\rangle$, the monomial ideal generated by monomials in $G_1, G_2 \ \text{and}\  G_3$.
	\end{enumerate}
\end{algorithm}

\begin{theorem}
	\label{thm.deg23odd3}
    The ideals constructed in Algorithms \ref{alg.odd3} and \ref{alg.odd2} are $f$-ideals.
\end{theorem}

\begin{proof}
Appealing to Lemma 2.7 and noting that $I$ only has generators in degree 2 and 3 (hence $D_{2}=\emptyset$ and $B_{3}=\emptyset$) we see that $I$ is an $f$-ideal if and only if
\[
2(\#C_{2})+\#B_{2}=\#\M_{n,2}\;\; \text{ and }\;\; 2(\#C_{3})+\#D_{3}=\#\M_{n,3}
\]
\[
\#B_{1}=\#\M_{n,1}\;\; \text{ and } \;\; \#D_{4}=\#\M_{n,4}
\]
 Counting the various subsets in algorithm 4.8 we get
\[\begin{tabular}{lllll}
$\#C_{2}=k^{2}-1$ &$\#B_{2}=k+2$ &$\#\M_{n,2}=2k^{2}+k$ &$\#B_{1}=\#\M_{n,1}$ \\
$\#C_{3}=\frac{k+1}{2}$ &$\#D_{3}=\frac{8k^{3}-8k-6}{6}$ &$\#\M_{n,3}=\frac{k(4k^{2}-1)}{3}$ &$\#D_{4}=\#\M_{n,4}$\\
\end{tabular}\] and see that the equalities hold.\\

Counting the various subsets in algorithm 4.9 we get
\[\begin{tabular}{lllll}
$\#C_{2}=k^{2}+k$ &$\#B_{2}=k+1$ &$\#\M_{n,2}=2k^{2}+3k+1$ &$\#B_{1}=\#\M_{n,1}$ \\
$\#C_{3}=\frac{k}{2}$ &$\#D_{3}=\frac{4k^{3}+6k^{2}-k}{3}$ &$\#\M_{n,3}=\frac{k(4k^{2}+6k+2)}{3}$ &$\#D_{4}=\#\M_{n,4}$\\
\end{tabular}\] and again see that the equalities hold.
\end{proof}

\begin{example} In Algorithm \ref{alg.odd3}, let $k=3$ and consider the ring $R=\kk[x_1,x_2,x_3,y_1,y_2,y_3,y_4]$. In this case $G_2=\{y_1y_2,y_1y_3,y_1y_4,y_2y_3,y_2y_4,y_3y_4\}$.  For  $l=1$ and $m=2$, $G_1=\{x_1x_3,x_2x_3\}$ and $M=\{x_1x_2y_1, x_1x_2y_2,x_1x_2y_3,x_1x_2y_4\}$. Now we can form the set $G_3$ in six different ways, which can be done by choosing any $2$-point subset of $M$. Let $G_3=\{x_1x_2y_1, x_1x_2y_2\}$. Then the output ideal
$I = \langle G_1 \cup G_2 \cup G_3\rangle$= $ \langle x_1x_3,x_2x_3, y_1y_2,y_1y_3,y_1y_4,y_2y_3,y_2y_4,y_3y_4, x_1x_2y_1, x_1x_2y_2 \rangle$ is an $f$-ideal of $R$ in mixed degrees with $f(\Delta(I)) =f(\Delta_{I})=(7,13,2)$.
\end{example}

\begin{example} In Algorithm \ref{alg.odd2}, let $k=4$ and consider the ring $R=\kk[x_1,x_2,x_3,x_4,y_1,y_2,y_3,$\\$y_4,y_5,y_6]$. In this case $G_1=\{x_1x_2, x_1x_3,x_1x_4,x_2x_3,x_2x_4,x_3x_4\}$. For  $l=1$ and $m=2$, we have $$G_2=\{y_1y_3,y_1y_4,y_1y_5,y_1y_6,y_2y_3,y_2y_4,y_2y_5,y_2y_6,y_3y_4,y_3y_5,y_3y_6,y_4y_5,y_4y_6,y_5y_6\}$$ and $M=\{x_1y_1y_2, x_2y_1y_2,x_3y_1y_2,x_4y_1y_2\}$. Now we can form the set $G_3$ in six different ways, which can be done by choosing any $2$-point subset of $M$. Let $G_3=\{x_1y_1y_2, x_2y_1y_2\}$. Then the output ideal
$$I = \langle G_1 \cup G_2 \cup G_3\rangle = \langle x_1x_2, x_1x_3,x_1x_4,x_2x_3,x_2x_4,x_3x_4, y_1y_3,y_1y_4,y_1y_5,y_1y_6,y_2y_3,y_2y_4,y_2y_5,$$ $$y_2y_6,y_3y_4,y_3y_5,y_3y_6,y_4y_5,y_4y_6,y_5y_6, x_1y_1y_2, x_2y_1y_2 \rangle$$ is an $f$-ideal of $R$ in mixed degrees with $f(\Delta(I)) =f(\Delta_{I})=(10,25,2)$.
\end{example}


\section{Pure $f$-ideals in small degrees} \label{sec.construction}

In this section we provide novel algorithms for constructing pure $f$-ideals in small degrees (up to degree 5). It is known that pure $f$-ideals exist in every degree $d$ if one is allowed enough variables. Our algorithms are effective for $n \ge d^2$.

For our algorithms in small degrees define the following partitions of variables:
$$ \text{For } i = 1, \dots, d \;\;S_i := \{x_j ~\big|~ j \equiv i \mod d\}$$

For a fixed partition of variables as above and a square-free monomial $g$ define \[b_{i}(g)=\#\left(\{x_{j}\}_{j\in\text{supp}(g)}\cap S_{i}\right)\] Let $a(g)=\left(a_{1}(g),\dots ,a_{d}(g)\right)$ where the entries $a_{j}$ are the $b_{i}(g)$ ordered from greatest to least. We can partition square-free monomials into the following sets indexed by $(\alpha_{1},\dots ,\alpha_{d})\in\mathbb{Z}^{d}$ \[\mathcal{A}(\alpha_{1},\dots, \alpha_{d})=\{g\in\cup_{i=1}^{\infty}\M_{n,i} ~\big|~ a(g)=(\alpha_{1},\dots , \alpha_{d})\}\]

\begin{algorithm}
	\label{alg.deg3}
Let $d = 3$ and $n \ge 9$.
\begin{enumerate}[align=left]
\item[\textsc{Step 1:}] Set $G_1 = \{ x_u x_v x_w ~\big|~ x_u, x_v, x_w\in S_i, i\in[3]\}$
\item[\textsc{Step 2:}] Set $G_2 = \{ x_u x_v x_w ~\big|~ x_u, x_v\in S_i, x_w\in S_{i+1}, i\in[3]\}$
\item[\textsc{Step 3:}] Select $\frac{1}{2}{n \choose 3}-\sum_{i=1}^{2}\#G_{i}$ monomials from $\M_{n,3}\backslash(G_{1}\cup G_{2})$ and call this set $G_{3}$
\item[\textsc{Step 4:}] Output $I=\left<G_{1}\cup G_{2}\cup G_{3}\right>$
\end{enumerate}
\end{algorithm}

\begin{theorem} The ideal $I$ constructed by Algorithm \ref{alg.deg3} is an $f$-ideal.
\end{theorem}

\begin{proof}
We appeal to Theorem \ref{thm.PerfectCharacterization}. To see $I$ is lower perfect note that in $\M_{n,2}$ monomials in $\mathcal{A}(2,0,0)$ divide some element of $G_{1}$, and monomials in $\mathcal{A}(1,1,0)$ divide some element of $G_2$. To see that $I$ is upper perfect note that in $\M_{n,4}$ all monomials in $\mathcal{A}(4,0,0)\cup\mathcal{A}(3,1,0)$ are divided by an element of $G_{1}$ and all monomials in $\mathcal{A}(2,2,0)\cup\mathcal{A}(2,1,1)$ are divided by an element of $G_{2}$. By construction the minimal number of generators of $I$ is $\frac{1}{2}{n \choose 3}$.
\end{proof}

\begin{example} For $n = 9$ and $d = 3$, an $f$-ideal must have precisely $\frac{1}{2}{9 \choose 3} = 42$ generators. We construct the sets
\vspace{-.1cm}
$$S_1 = \{x_1, x_4, x_7\}\;\;\; S_2 = \{x_2, x_5, x_8\}\;\;\; S_3 = \{x_3, x_6, x_9\}$$
\vspace{-.4cm}
\[
\begin{array}{llll}
&G_{1}= &\{x_1 x_4 x_7, x_2 x_5 x_8, x_3 x_6 x_9\}\\
&G_{2} = &\{x_1 x_4 x_2, x_1 x_4 x_5, x_1 x_4 x_8, x_1 x_7 x_2, x_1 x_7 x_5, x_1 x_7 x_8, x_4 x_7 x_2, x_4 x_7 x_5, x_4 x_7 x_8, x_2 x_5 x_3,\\
& &x_2 x_5 x_6, x_2 x_5 x_9, x_2 x_8 x_3, x_2 x_8 x_6, x_2 x_8 x_9, x_5 x_8 x_3, x_5 x_8 x_6, x_5 x_8 x_9, x_3 x_6 x_1, x_3 x_6 x_4,\\
& &x_3 x_6 x_7, x_3 x_9 x_1, x_3 x_9 x_4, x_3 x_9 x_7, x_6 x_9 x_1, x_6 x_9 x_4, x_6 x_9 x_7\}
\end{array}
\]

Note that $\#G_{1} = 3$ and $\#G_{2} = 27$. Hence we must select an additional $12$ generators not in $G_1$ or $G_2$ to construct $G_{3}$. These can be chosen or randomly generated. For example, let \[G_3 = \{x_1 x_2 x_3, x_1 x_2 x_6, x_1 x_2 x_9, x_1 x_5 x_3, x_1 x_5 x_6, x_1 x_5 x_9, x_1 x_8 x_3, x_1 x_8 x_6, x_1 x_8 x_9, x_4 x_5 x_6, x_4 x_5 x_9, x_7 x_8 x_9 \}\] Then $I = \left<G_{1}\cup G_{2}\cup G_{3}\right>$ is an $f$-ideal.
\end{example}

\begin{algorithm}
	\label{alg.deg4}
Let $d = 4$ and $n \ge 16$.
\begin{enumerate}[align=left]
\item[\textsc{Step 1:}] Set $G_1 = \{x_t x_u x_v x_w ~\big|~ \{x_t, x_u\} \subset S_i, \{x_v\} \subset S_{i+1}, \{x_w\} \subset S_{i+2}, i\in[4]\}$
\item[\textsc{Step 2:}] Set $G_2 = \{x_t x_u x_v x_w ~\big|~ \{x_t, x_u\} \subset S_i, \{x_v, x_w\} \subset S_j, 1 \leq i < j \leq 4 \}$
\item[\textsc{Step 3:}] Set $G_3 = \{x_t x_u x_v x_w ~\big|~ \{x_t, x_u, x_v\} \subset S_i, \{x_w\} \subset s_{i-1}, i\in[4] \}$
\item[\textsc{Step 4:}] Set $G_4 = \{ x_t x_u x_v x_w ~\big|~ \{x_t, x_u, x_v, x_w\} \subset S_i, i\in[4] \}$
\item[\textsc{Step 5:}]
Select $\frac{1}{2}{n \choose 4}-\sum_{i=1}^{4}\#G_{i}$ monomials from $\M_{n,4}\backslash(G_{1}\cup\cdots\cup G_{4})$ and call this set $G_{5}$.
\item[\textsc{Step 6:}] Output $I=\left<G_{1}\cup\cdots\cup G_{5}\right>$
\end{enumerate}
\end{algorithm}

\begin{theorem} The output from Algorithm \ref{alg.deg4} is an $f$-ideal.
\end{theorem}

\begin{proof}

We again appeal to Theorem \ref{thm.PerfectCharacterization}. To see $I$ is lower perfect note that in $\M_{n,3}$ monomials in $\mathcal{A}(3,0,0,0)$ divide some element of $G_{4}$, and monomials in $\mathcal{A}(2,1,0,0)$ divide some element of $G_{2}$ and monomials in $\mathcal{A}(1,1,1,0)$ divide some element of $G_{1}$. To see that $I$ is upper perfect note that in $\M_{n,5}$ all monomials in $\mathcal{A}(5,0,0,0)\cup\mathcal{A}(4,1,0,0)$ are divided by an element of $G_{4}$, all monomials in $\mathcal{A}(2,1,1,0)$ are divided by either $G_{1}$ or $G_{3}$, all monomials in $\mathcal{A}(1,1,1,1)$ are divided by an element of $G_{1}$, all monomials in $\mathcal{A}(3,2,0,0)$ are divided by elements of $G_{2}$. By construction the minimal number of generators of $I$ is $\frac{1}{2}{n \choose 4}$.
\end{proof}

\begin{algorithm}
	\label{alg.deg5}
Let $d=5$ and $n \ge 25$.
\begin{enumerate}[align=left]
\item[\textsc{Step 1:}] Set $G_{1} = \{ x_s x_t x_u x_v x_w ~\big|~ x_s, x_t\in S_i, x_u\in S_{i+1}, x_v\in S_{i+2}, x_w\in S_{i+3}, i\in[5] \}$
\item[\textsc{Step 2:}] Set $G_{2} = \{ x_s x_t x_u x_v x_w ~\big|~ x_s, x_t\in S_i, x_u, x_v\in S_{i-1}, x_w\in S_{i-2}, i\in[5] \}$
\item[\textsc{Step 3:}] Set $G_3 = \{ x_s x_t x_u x_v x_w ~\big|~ x_s, x_t\in S_i, x_u, x_v\in S_{i+2}, x_w\in S_{i+4}, i\in[5] \}$
\item[\textsc{Step 4:}] Set $G_4 = \{ x_s x_t x_u x_v x_w ~\big|~ x_s, x_t, x_u\in S_i, x_v\in S_{i-1}, x_w\in S_{i-3}, i\in[5] \}$
\item[\textsc{Step 5:}] Set $G_5 = \{ x_s x_t x_u x_v x_w ~\big|~ x_s, x_t, x_u\in S_i, x_v\in S_{i+1}, x_w\in S_{i+3}, i\in[5] \}$
\item[\textsc{Step 6:}] Set $G_6 = \{ x_s x_t x_u x_v x_w ~\big|~ x_s, x_t, x_u\in S_i, x_v,x_w\in S_{i-1}, i\in[5]\}$
\item[\textsc{Step 7:}] Set $G_7 = \{ x_s x_t x_u x_v x_w ~\big|~ x_s, x_t, x_u\in S_i, x_v,x_w\in S_{i-3}, i\in[5]\}$
\item[\textsc{Step 8:}] Set $G_8 = \{ x_s x_t x_u x_v x_w ~\big|~ x_s, x_t, x_u, x_v\in S_i, x_w\in S_{i+1}, i\in[5] \}$
\item[\textsc{Step 9:}] Set $G_9 = \{ x_s x_t x_u x_v x_w ~\big|~ x_s, x_t, x_u, x_v\in S_i, x_w\in S_{i+3}, i\in[5] \}$
\item[\textsc{Step 10:}] Set $G_{10} = \{ x_s x_t x_u x_v x_w ~\big|~ x_s, x_t, x_u, x_v, x_w\in S_i, i\in[5] \}$
\item[\textsc{Step 11:}] Select $\frac{1}{2}{n \choose 5}-\sum_{i=1}^{10}\#G_{i}$ monomials from $\M_{n,5}\backslash(G_{1}\cup\cdots\cup G_{10})$ and call this set $G_{11}$.
\item[\textsc{Step 12:}] Output $I=\left<G_{1}\cup\cdots\cup G_{11}\right>$
\end{enumerate}
\end{algorithm}

\begin{theorem} The output from Algorithm \ref{alg.deg5} is an $f$-ideal.
\end{theorem}

\begin{proof}

We again appeal to Theorem \ref{thm.PerfectCharacterization}. To see $I$ is lower perfect note that in $\M_{n,4}$ we note that monomials in $\mathcal{A}(1,1,1,1,0)$ divide some element of $G_{1}$, monomials in $\mathcal{A}(2,1,1,0,0)$ divide some element of $G_{1}\cup G_{2}\cup G_{3}$, monomials in $\mathcal{A}(2,2,0,0,0)$ divide some element of $G_{2}\cup G_{3}$, monomials in $\mathcal{A}(3,1,0,0,0)$ divide some element of $G_{4}\cup G_{5}$, and monomials in $\mathcal{A}(4,0,0,0,0)$ divide some element of $G_{10}$. To see that $I$ is upper perfect note that in $\M_{n,6}$ monomials in $\mathcal{A}(2,1,1,1,1)$ are divided by some element of $G_{1}$, monomials in $\mathcal{A}(2,2,1,1,0)\cup\mathcal{A}(3,1,1,1,0)$ are divided by some element of $G_{1}\cup G_{2}\cup G_{3}\cup G_{4}\cup G_{5}$, monomials in $\mathcal{A}(2,2,2,0,0)$ are divided by some element in $G_{2}\cup G_{3}$, monomials in $\mathcal{A}(3,2,1,0,0)$ are divided by some element of $G_{2}\cup G_{3}\cup G_{4}\cup G_{5}$, monomials in $\mathcal{A}(4,1,1,0,0)$ are divided some element of $G_{4}\cup G_{7}\cup G_{8}$, monomials in $\mathcal{A}(5,1,0,0,0)\cup\mathcal{A}(4,2,0,0,0)\cup\mathcal{A}(3,3,0,0,0)$ are divided by some element in $G_{6}\cup G_{7}\cup G_{8}\cup G_{9}\cup G_{10}$, and monomials in $\mathcal{A}(6,0,0,0,0)$ are divided by some element of $G_{10}$. By construction the minimal number of generators of $I$ is $\frac{1}{2}{n \choose 5}$.

\end{proof}

We end the paper with an example and some questions. The following example of a mixed $f$-ideal in three degrees (4,5,6) was discovered by Adam Van Tuyl: Let $R = \mathbb{K}[x_1..x_8]$ and define an ideal $I\subset R$
\begin{align*}
I = &(x_{1}x_{2}x_{3}x_{8},x_{2}x_{3}x_{5}x_{8}, x_{3}x_{4}x_{5}x_{8}, x_{1}x_{2}x_{4}x_{5}x_7,  x_{2}x_{3}x_{4}x_{5}x_{7}, x_{3}x_{4}x_{5}x_{8}, \\
& x_{2}x_{3}x_{6}, x_{1}x_{2}x_{4}x_{6},  x_{1}x_{2}x_{5}x_{6}, x_{1}x_{3}x_{4}x_{6}, x_{1}x_{3}x_{5}x_{6}, x_{1}x_{4}x_{5}x_{6}, x_{2}x_{3}x_{4}x_{6}, \\
& x_{2}x_{3}x_{5}x_{6}, x_{2}x_{4}x_{5}x_{6}, x_{6}x_{7}x_{8}, x_{5}x_{6}x_{7}x_{8}, x_{4}x_{6}x_{7}, x_{4}x_{7}x_{8}, x_{1}x_{5}x_{8}, \\
& x_{2}x_{4}x_{6}x_{8}, x_{4}x_{5}x_{6}x_{8}, x_{3}x_{4}x_{6}x_{8}, x_{2}x_{5}x_{6}x_{8}
 x_{1}x_{3}x_{4}x_{5}, x_{3}x_{5}x_{6}x_{8},  x_{3}x_{6}x_{7}x_{8}, \\
& x_{3}x_{7}x_{8}, x_{2}x_{7}x_{8}, x_{1}x_{6}x_{8},x_{2}x_{7}x_{8}, x_{2}x_{5}x_{6}x_{8})
\end{align*}
A quick check in Macaulay2 verifies that this is indeed an $f$-ideal. One may ask:
\begin{question}
Can we always find $f$-ideals generated in consecutive degrees? Could one create an algorithm to construct such $f$-ideals? Are there $f$-ideals generated in nonconsecutive degrees?
\end{question}
Our density result,Theorem \ref{thm.density}, assumes that d is fixed while $n\rightarrow \infty$ so one may also naturally ask:
\begin{question} \quad
 Is the asymptotic density of $\fideal(n,d)$, as $n\rightarrow \infty$, still 0 if we allow $d$ to vary with $n$?
\end{question}
Finally, we managed to come up with examples of $f$-ideal generated in degrees 2 and 3 which are seemingly unrelated. It seems likely that these are far from the only examples and so one could ask:
\begin{question}
 Is there a classification of mixed $f$-ideals in degree $2$ and $3$ similar to the one for pure $f$-ideals given in Theorem 3.5 \cite{deg2}?
\end{question}


\end{document}